\documentclass[reqno,10pt]{amsart}
\usepackage{amscd,amssymb,amsmath,color,url,stmaryrd}
\usepackage{extarrows}
\usepackage{latexsym}
\usepackage[all]{xy}
\usepackage{defs}
\usepackage[colorlinks=true]{hyperref}
\usepackage{tikz-cd}
\usepackage{forest}

\def\cAcirc{\calAcirc}
\def\sp{{\rm sp}}
\def\..{{,\ldots,}}
\def\Rad{{\rm Rad}}

\def\NN{{\bbN}}

\def\QQ{{\bbQ}}
\def\RR{{\bbR}}

\def\cM{{\calM}}
\def\cA{{\calA}}
\def\cB{{\calB}}

\def\cC{{\calC}}

\def\cH{{\calH}}

\def\cM{{\calM}}

\def\dim{{\rm dim}}

\begin{document}
\title{Geometrically multiplicative non-archimedean norms}
\author{Michael Temkin}
\address{Einstein Institute of Mathematics\\
               The Hebrew University of Jerusalem\\
                Edmond J. Safra Campus, Giv'at Ram, Jerusalem, 91904, Israel}
\email{michael.temkin@mail.huji.ac.il}

\author{Grigory Zutler}
\address{Einstein Institute of Mathematics\\
               The Hebrew University of Jerusalem\\
                Edmond J. Safra Campus, Giv'at Ram, Jerusalem, 91904, Israel}
\email{grigory.zutler@mail.huji.ac.il}

\thanks{This research was supported by MPIM-Bonn and ISF Grant 1203/22}

\begin{abstract}
Universally (or geometrically) multiplicative norms on Banach algebras over complete non-archimedean fields were used by Berkovich in his works on non-archimedean geometry, and later they were studied in some detail by Poineau. In this paper, we perform a more thorough study of the conditions under which multiplicativity, spectrality and spectral multiplicativity of norms on algebras over real valued fields are preserved by ground field extensions. We obtain precise criteria quite analogous to the classical theory of geometric irreducibility and reducedness. In particular, we generalize the results of Poineau in a few aspects.
\end{abstract}

\maketitle

\setcounter{tocdepth}{1}
\tableofcontents

\section{Introduction}

\subsection{Background and motivation}
All rings in this paper are assumed to be commutative and all seminorms are non-archimedean. Banach rings and their spectra are building blocks of Berkovich analytic geometry. Our goal is to study finer properties of norms, not preserved by equivalence of Banach rings. In addition it is often convenient to work with non-complete rings and apply completion only when needed, so this paper is written in the generality of normed or even seminormed rings over (non-archimedean) real valued fields. One can always divide a seminormed ring or module by the kernel of the seminorm and pass to the normed world, but seminorms show up naturally as outcomes of such constructions as the spectral seminorm and the tensor seminorm. For this reason we prefer to consider the seminormed case too.

Here are three very important properties a seminorm $|\ |$ can satisfy:
\begin{itemize}
\item[(i)] $|\ |$ is {\em multiplicative} if $|1|=1$ and $|ab|=|a|\cdot|b|$ for any $a,b\in\cA$.
\item[(ii)] $|\ |$ is {\em power-multiplicative} or {\em spectral} if $|1|=1$ and $|a^n|=|a|^n$ for any $a\in\cA$, $n\ge 1$.
\item[(iii)] $|\ |$ is {\em spectrally multiplicative} if $|\ |_\sp$ is multiplicative.
\end{itemize}
Recall that the spectral seminorm $|\ |_\sp$ defined by $|a|_\sp=\lim_n|a^n|^{1/n}$ is the maximal power-multiplicative seminorm dominated by $|\ |$, so condition (ii) means that $|\ |=|\ |_\sp$.

When working over a real valued ground field $k$ it is also natural to consider geometric or universal variants of these notions. Namely, we say that a seminorm on a $k$-algebra $\cA$ is {\em geometrically} multiplicative, spectral or spectrally multiplicative if the tensor seminorm on $\cA_l=\cA\otimes_kl$ is multiplicative, spectral or spectrally multiplicative, respectively, for any real valued field extension $l/k$. Geometrically multiplicative norms were used by Berkovich to study group actions on analytic spaces, and he called the points they define {\em peaked points} in \cite{berbook}. Poineau called them {\em universal norms} in \cite{Poineau} and proved that any multiplicative norm over an algebraically closed field is universal. His result was reproved by model theoretic tools in \cite{BY}. To the best of our knowledge no further study of universality was done in the literature. The goal of this paper is to fill in this gap and develop the theory of geometric spectrality and spectral multiplicativity quite analogous to the classical theory of geometric reducedness and irreducibility of schemes. The combination of these two yields the desired criterion of geometric multiplicativity.

\subsection{Normed algebra}
We will work in the category of normed (or even seminormed) rings (and modules) with non-expansive homomorphisms. Seminormed rings will be denoted by calligraphic letters, e.g. $\cA=(\cA,|\ |)$, and we say that $\cA$ is {\em spectral}, {\em multiplicative} or {\em spectrally multiplicative} if its seminorm is so. In fact, these properties should be viewed as extensions to the theory of seminormed rings of the classical properties from commutative algebra -- reducedness, being a domain and having an integral reduction. This analogy will be used throughout the paper, but here are a couple of its instances.

\begin{itemize}
\item[(o)] A ring $A$ is reduced (resp. integral) if and only if the norm sending all non-zero elements to 1 is power-multiplicative (resp. multiplicative).

\item[(i)] A norm on $\cA$ is power-multiplicative (resp. multiplicative) if and only if the associated graded ring $\cA_\gr=\oplus_{r>0}\cA_{\le r}/\cA_{<r}$ is reduced (resp. integral).

\item[(ii)] Each normed ring $\cA$ possesses a universal homomorphism $\phi\:\cA\to\cA^\sp$ whose target is a spectral normed ring $\cA^\sp$ that will be called the {\em spectralization} of $\cA$. Furthermore, $\calM(\cA^\sp)=\calM(\cA)=X$, $|\ |_{\cA^\sp}=\max_{x\in X}|\ |_x$ and the kernel of $\phi$ consists of all quasi-nilpotent elements, i.e. elements with $|f|_\sp=\max_{x\in X}|f|_x=0$. This is the analogue of the reduction homomorphism $A\to\tilA=A/\Rad(A)$ in commutative algebra and of the fact that $\Rad(A)=\cap_{p\in\Spec(A)}p$ and the reduction induces a homeomorphism of spectra.

\item[(iii)] A seminormed ring $\cA$ is spectrally multiplicative if and only if $\cM(\cA)$ contains a unique maximal point $|\ |_x$. A seminormed ring is multiplicative if and only if it is spectral and spectrally multiplicative. This is the analogue of the fact that a ring $A$ is a domain if and only if it is reduced and $\Spec(A)$ is irreducible.
\end{itemize}

\subsection{Main results}
We will study criteria for normed $k$-algebras to be geometrically spectral and geometrically spectrally multiplicative. It turns out that the theory is very analogous to its classical commutative algebra counterpart. For instance, if $k$ is of equal characteristic $p$, then the seminorm of $\cA$ is geometrically spectral if and only if the tensor seminorm on $\cA\otimes k^{1/p}$ is spectral. Moreover, one can replace $k^{1/p}$ by its deformations, called $p$-versal extensions of $k$, and then a similar criterion also applies in the mixed characteristic, see Theorem~\ref{spectralth}. Also, we prove that $\cA$ is geometrically spectrally multiplicative if and only if it is spectrally multiplicative and $\hatk$ is separably closed in the completion of $\Frac(\cA^\sp)$, see Theorem~\ref{multiplicativeth}. Combining these two results one deduces in Theorem~\ref{multth} a criterion for $\calA$ to be geometrically multiplicative. In particular, we reprove the result of Poineau and show the stronger result that over a perfectoid base $k$ any spectral $k$-algebra $\cA$ is geometrically spectral and geometric multiplicativity holds if and only if $\cA$ is multiplicative and $k$ is algebraically closed in the completion of $\Frac(\cA)$.

We will see that the main case to study is when $l/k$ is finite. If $l/k$ is defectless, then the graded reduction completely controls the ground field extension and can be used to formally reduce the problem to the classical commutative algebra. However in the general case, instead of this we will have to work with reductions of a small but non-zero thickness $\cAcirc/\pi\cAcirc$ for a pseudo-uniformizer $\pi\in k$ whose valuation is close enough to 1. This imposes mild technical complications as we have to track the thickness (e.g. under Frobenius), but allows us to exploit the analogy with the classical results, and construct the proofs along the same general lines.

\subsection{Conventions}
For simplicity, we will omit the word ``generalized'' in what one often calls generalized Gauss valuation or extension. By default, we provide tensor products of seminormed rings and modules with tensor product seminorm. If $l/k$ is a real valued field extension and $\cA$ is a seminormed $k$-algebra, we will use the notation $\cA_l=\cA\otimes_kl$.

\subsection{AI disclosure}
GPT-6 Astra was used when working on the revision of the first arXiv version. It pointed out numerous inaccuracies that were corrected.

\setcounter{tocdepth}{1}

\tableofcontents

\section{Analytic spectrum}

\subsection{Berkovich analytic spectrum}
Berkovich introduced analytic spectrum only for Banach rings, though this is mainly a matter of taste. As in the theory of Huber's adic spectra, the definition makes sense more generally, but the outcome depends only on the completion viewed as a Banach ring. Thus, by Berkovich spectrum of a seminormed ring $\cA$ we mean the set $X=\cM(\cA)$ of all bounded real semivaluations on $\cA$ with the weakest topology making all functions $|f|\:X\to\RR_{\ge 0}$ with $f\in\cA$ continuous. This is a contravariant functor on the category of seminormed rings.

The points of $X$ will be denoted $x$ or $|\ |_x$. The completed residue field $\cH(x)$ is the completed field of fractions of the multiplicative normed ring $\cA/\Ker(|\ |_x)$. Given a homomorphism $\cA\to\cB$ of seminormed rings we say that the induced map $\cM(\cB)\to\cM(\cA)$ is an {\em isomorphism}\footnote{The notion {\em spectral isomorphism} would be more precise, but we do not provide $\cM(\cA)$ with any finer structure in this paper.} if it is a homeomorphism and induces isomorphisms of completed residue fields. For example, the map $\cM(\hatcalA)\to\cM(\cA)$ is an isomorphism because of the factorizations $\cA\to\hatcalA\to\calH(x)$. Thus Berkovich's theorem  \cite[Theorem~1.2.1]{berbook} implies that $\cM(\cA)$ is compact and non-empty whenever seminorm is not zero, i.e. $\hatcalA\neq 0$.

\subsection{Spectralization}
By the {\em spectralization} of $\cA$ we mean the normed ring $\cA^\sp=\cA/\Ker(|\ |_\sp)$, whose norm is induced by the spectral seminorm $|\ |_\sp$ of $\cA$ and denoted by the same notation. This procedure is the analogue of reduction in commutative algebra. Recall that $|\ |_\sp=\max_x|\ |_x$ by \cite[Theorem~1.3.1]{berbook}, which is the analogue of the classical fact that the nilradical equals the intersection of all prime ideals.

\begin{lem}\label{prelem}
(i) For any seminormed ring $\cA$ the map $\cA\to\cA^\sp$ induces an isomorphism of spectra.

(ii) For any seminormed $\cA$-algebras $\cB$ and $\cC$ one has that $$(\cB\otimes_\cA\cC)^\sp=(\cB^\sp\otimes_\cA\cC^\sp)^\sp.$$

(iii) A seminormed ring $\cA$ is spectrally multiplicative if and only if $\cM(\cA)$ possesses a single maximal point $x$ (i.e. $|\ |_x\ge|\ |_y$ for any $y\in\cM(\cA)$). In this case, $|\ |_\sp=|\ |_x$ and $\cH(x)$ is the completed fraction field of $\cA^\sp$.
\end{lem}
\begin{proof}
The first claim is obvious, the second one follows from the simple observation that the product norm $|\ |_{\cB,\sp}\otimes|\ |_{\cC,\sp}$ dominates any power-multiplicative seminorm on $\cB\otimes_\cA\cC$ dominated by the original tensor seminorm. Finally, (iii) follows from the cited above fact that $|\ |_\sp=\max_x|\ |_x$.
\end{proof}

As a consequence, studying geometric spectral multiplicativity reduces to the case of fields.

\begin{cor}\label{precor}
Assume that $\cA$ is a spectrally multiplicative seminormed $k$-algebra over a real valued field $k$, and let $K=\Frac(\cA^\sp)$ with the induced valuation. Let $l/k$ be a real valued extension, then $\cA_l$ is spectrally multiplicative if and only if $K_l=K\otimes l$ is spectrally multiplicative. In particular, $\cA$ is geometrically spectrally multiplicative if and only if $K$ is.
\end{cor}
\begin{proof}
By Lemma~\ref{prelem} $\cA_l^\sp=(\cA^\sp\otimes_k l)^\sp$, so $\cA_l$ is spectrally multiplicative if and only if $\cA^\sp\otimes_k l$ is. Therefore it is enough to prove the claim when $\cA=\cA^\sp$, that is, the norm of $\cA$ is multiplicative. It is a classical result (and also follows from Lemma~\ref{isomlem} below) that $\cA\into\cA_l$ is an isometry, and hence this is also true for $\cA\into(\cA_l)^\sp$.

There is a canonical isometric embedding $\cA\into K$, so by Lemma~\ref{isomlem} the embedding $\cA_l\into K_l$ is an isometry and hence also the spectral seminorms agree. Thus, if $K_l$ is spectrally multiplicative, so is $\cA_l$. The converse is also true because any $a\neq 0$ in $\cA$ satisfies $|a|\cdot|a^{-1}|=1$ in $K$, and hence for any $x\in\cA_l$ we have that $|x/a|_{K_l,\sp}=|x|_{\cA_l,\sp}/|a|$.
\end{proof}

\subsection{Unibranched extensions}
Our choice to work in the generality of non-complete and even non-henselian valued fields forces us to distinguish abstract algebraic extensions of real valued fields, which may have a few extensions of valuations, and real valued extensions, where an extension is fixed. We say that an algebraic extension $l/k$ is {\em unibranched} if the extension is unique. Also, we say that a finite extension of real valued fields $l/k$ is {\em defectless} if $[l:k]=e_{l/k}f_{l/k}$. In particular, in this case $l/k$ is unibranched. This should not be confused with the more usual definition that an abstract extension $l/k$ is defectless when $[l:k]=\sum_i e_if_i$, where the sum is over all extensions of the valuation.

\begin{exam}
Assume that $k$ is a real valued field and $l/k$ a finite extension. Let $|\ |_1\..|\ |_r$ be all extensions of the valuation of $k$ to $l$, and let $l_i=(l,|\ |_i)$ be the corresponding real valued fields. Provide $l$ with any $k$-norm $\|\ \|$ which is equivalent to a cartesian norm (as recalled below). Then $\|\ \|_\sp$ is the classical spectral norm $\max_i|\ |_i$ and hence the spectral completion is $\hatl=\prod_i\hatl_i$ and $\cM(l)=\coprod_i\cM(l_i)=\coprod_i\cM(\hatl_i)$. This is essentially equivalent (and follows from) the classical fact from commutative algebra that the integral closure of $\kcirc$ in $l$ is $\lcirc=\cap_i\lcirc_i$, and it is the unit ball of $\|\ \|_\sp$.

The extension is unibranched if and only if $\|\ \|_\sp$ is a valuation. From the analytic point of view non-unibranched abstract extensions (with the spectral norm) are non-local objects, as opposed to a real-valued extension $l_i/k$, where a specific extension is fixed. If $l/k$ is not unibranched, $l$ is not $k$-cartesian with respect to any single valuation $|\ |_i$. The branches are separated analytically or \'etale-locally: $l^h=k^h\otimes l=\prod_i l_i^h$, where $l^h$ is the total ring of fractions of the henselization of $\lcirc$. The henselian factorization is finer, because $\hatl$ is the quotient of $\hatk\otimes l$ by the kernel of the spectral seminorm, which can be non-trivial.
\end{exam}

\subsection{Graded reduction and orthogonality}
Recall that a vector space $V$ over a real valued field $k$ is called {\em cartesian} if it possesses an {\em orthogonal} basis $\{e_i\}_{i\in I}$, which means that $\|\sum_i c_ie_i\|=\max_i\|c_ie_i\|$. A finite real-valued extension $l/k$ is defectless if and only if $[\till_\gr:\tilk_\gr]=[l:k]$ if and only if $l$ is cartesian, where $\tilk_\gr=\oplus_{r>0}\tilk_r=\oplus_{r>0}k_{\le r}/k_{<r}$ is the {\em graded reduction} of $k$ as defined in \cite[\S2]{Temkin-local-properties}. Recall that $\tilk_\gr$ is a graded field (i.e. 0 and (1) are the only homogeneous ideals). More generally, note that for any normed $k$-vector space $V$ the graded reduction $\tilV_\gr$ is a graded $\tilk_\gr$-vector space and the same argument as for $l/k$ shows that if $V$ is finite-dimensional, then it is cartesian if and only if the fundamental inequality $\dim_{\tilk_\gr}(\tilV_\gr)\le \dim_k(V)$ is an equality.

\begin{lem}\label{grlem0}
Let $V,W$ be normed vector spaces over a real valued field $k$, and assume that $V$ is finite-dimensional cartesian. Then $U=V\otimes W$ satisfies $\tilU_\gr=\tilV_\gr\otimes_{\tilk_\gr}\tilW_\gr$, and if $W=K$ is a real valued extension of $k$, then $U=V_K=V\otimes_kK$ is a cartesian $K$-vector space.
\end{lem}
\begin{proof}
Choose an orthogonal basis $v_1\..v_n$ of $V$ and set $U_i=W\otimes v_i$. Then a direct inspection shows that the decomposition $U=\oplus_i U_i$ is orthogonal and hence $\tilU_\gr$ is the direct sum of $\widetilde{(U_i)}_\gr$, which easily implies the claim.
\end{proof}

\subsection{Applications to cartesian base changes}
As a consequence one can easily control the norms under finite defectless base changes. This case already covers complete discrete valuations and finite tame extensions.

\begin{lem}\label{grlem}
Let $l/k$ be a finite defectless extension of real valued fields. Then for any real valued $k$-field $K$ and $L=K\otimes l$ the following claims hold:

(i) $L$ is spectral if and only if $\tilL=\tilK_\gr\otimes_{\tilk_\gr}\till_\gr$ is reduced. Furthermore, in this case $\tilL=\prod_{i=1}^r\tilL_i$, where each $\tilL_i$ is the graded reduction of $L_i=(L,|\ |_i)$ and $\{|\ |_1\..|\ |_r\}=\cM(L)$ is the set of all $K$-valuations on $L$ (in particular, $\hatL=\hatK\otimes_kl=\prod_{i=1}^r\hatL_i$).

(ii) $L$ is multiplicative if and only if $\tilK_\gr\otimes_{\tilk_\gr}\till_\gr$ is a graded field if and only if (i) holds with $r=1$.

(iii) More generally, for any normed $k$-algebra $\mathcal A$ the algebra
$\mathcal A_l=\mathcal A\otimes_k l$ is spectral if and only if
$\widetilde{\mathcal A}_{\rm gr}\otimes_{\widetilde k_{\rm gr}}\widetilde l_{\rm gr}$
is reduced.
\end{lem}
\begin{proof}
We start with (iii). By Lemma \ref{grlem0} $\widetilde{({\mathcal{A}_l})}_\gr=\widetilde{\mathcal{A}}_\gr\otimes_{\widetilde{k}_\gr}\widetilde{l}_\gr$, and by \S1.2(i) the tensor norm $\|\ \|$ is spectral (resp. multiplicative) if and only if $\widetilde{\mathcal{A}_l}$ is reduced (resp. integral). In (i) and (ii), $\calA=K$ is a field with $\tilL=\tilK_\gr\otimes_{\tilk_\gr}\till_\gr$. So, if $\tilL$ is reduced, then $\|\ \|=\max_i|\ |_i$ and hence $\hatL=\prod_i\hatL_i$ and the graded reduction is $\tilL=\prod_{i=1}^r\tilL_i$.
\end{proof}

\section{Almost orthogonality}
In the presence of defect one has to consider bases which are close enough to orthogonal ones. Of course this is more technically involved, but allows to simplify one aspect -- if the valuation is discrete one has to use the graded reduction, while otherwise it suffices to work with the unit balls only.

\subsection{Weak cartesianity}
The definition makes sense for seminormed rings and modules, but let us restrict the generality to a seminormed vector space $V$ over a real valued field $k$. A family of elements $\{e_i\}_{i\in I}$ is called $r$-orthogonal for $r\in(0,1]$ if $\|e_i\|\neq 0$ for $i\in I$ and for any linear combination $v=\sum_{i\in I}a_ie_i$ (with almost all $a_i$ zeros) one has that $\|v\|\ge r\max_i\|a_ie_i\|$. We say that $V$ is {\em $r$-cartesian} if it possesses an $r$-orthogonal basis, and {\em weakly cartesian} if any finite dimensional subspace possesses such a basis. In particular, $V$ is normed. As earlier, if $r=1$ we simply say {\em orthogonal} and {\em cartesian}. The notion of $r$-cartesianity is especially important in the finite-dimensional case when the valuation is not discrete. In particular, one has the following classical result, e.g. see \cite[Propositions~2.3.3/6 and 2.6.2/3]{bgr}.

\begin{lem}\label{cartlem}
Let $k$ be a real valued field and let $V$ be a finite-dimensional normed $k$-vector space. The following conditions are equivalent:
\begin{itemize}
\item[(i)] $V$ is $r$-cartesian for some $r>0$.
\item[(ii)] $V$ is $r$-cartesian for any $r\in(0,1)$.
\item[(iii)] The norm of $V$ is equivalent to a cartesian norm.
\item[(iv)] The map $V\otimes_k\hatk\to\hatV$ is injective.
\item[(v)] $\dim_k(V)=\dim_{\hatk}(\hatV)$.
\end{itemize}
If $k$ is discretely valued this is also equivalent to $V$ being cartesian.
\end{lem}

\subsection{Unit balls}
Given a normed $k$-vector space $V$ we will use the notation $\Vcirc=V_{\le 1}$ to denote its unit ball. As usual, we say that a $\kcirc$-module $M$ {\em almost} vanishes if $|k^\times|$ is dense and $\pi M=0$ for any $\pi\in\kcirccirc$.

\begin{lem}\label{reeslem}
Let $k$ be a real valued field with dense group of values $|k^\times|$, let $U,V$ be seminormed vector $k$-spaces and $r\in(0,1]$.

(i) An embedding $U\into V$ is an isometry if and only if the $\kcirc$-module $\Vcirc/\Ucirc$ is torsion free if and only if $\Vcirc/\Ucirc$ is almost torsion free.

(ii) Let $v_1\..v_n\in V$ be elements with $r_i=\|v_i\|>0$. Consider the normed vector space $W=\oplus_{i=1}^n ke_i=\oplus_{i=1}^n k(r_i)$ with orthogonal basis $e_1\..e_n$ such that $\|e_i\|=r_i$. Then the elements $v_1\..v_n$ are $r$-orthogonal if and only if the map $\phi\:W\to V$ taking $e_i$ to $v_i$ is injective and the torsion part of the cokernel of the induced map $\phicirc\:\Wcirc\to\Vcirc$ is annihilated by any $\pi$ with $|\pi|<r$. If this holds, then $\Coker(\phicirc)_\tor$ is killed also by any $\pi$ with $|\pi|=r$.

(iii) Let $W=U\otimes V$ be provided with the tensor seminorm. Then $\Ucirc\otimes_{\kcirc}\Vcirc\subseteq\Wcirc$ and the cokernel almost vanishes.
\end{lem}
\begin{proof}
All three claims reduce to straightforward unravelling the definitions. For example, let us check (iii). We should prove that $\pi w\in\Ucirc\otimes_{\kcirc}\Vcirc$ for any $w\in\Wcirc$ and $\pi\in\kcirccirc$. By definition of the tensor seminorm there exists a presentation $\pi w=\sum_i u_i\otimes v_i$ with $\|u_i\|\cdot\|v_i\|<1$ for any $i$. Since $|k^\times|$ is dense, we can find $c_i\in k$ such that $\|c_iu_i\|<1$ and $\|c_i^{-1}v_i\|<1$ and hence $\pi w\in\Ucirc\otimes_{\kcirc}\Vcirc$.
\end{proof}

Unlike claims (i) and (ii) of the lemma, one can not completely remove the adjective ``almost'' in claim (iii), and this is so even when the vector spaces are strict (i.e. $\|U\|=|k|$) or, moreover, $|k|=\bbR_{\ge 0}$. Here are some examples.

\begin{exam}
Assume for simplicity that $k$ is algebraically closed and complete, but not spherically complete, and choose a decreasing sequence of balls $B_i=B(a_i,r_i)$ without common $k$-points. In particular, $r=\lim_i r_i>0$ and usually one can achieve that $r$ is arbitrary. 

The intersection $\cap_i B_i$ in the Berkovich affine line is a single point $x$ of type 4. The completed residue field $K=\cH(x)$ is an immediate extension of $k$ with the valuation determined by the inequalities $|t-a_i|\le r_i$, hence $K$ is strict over $k$. In fact $K=\widehat{k[t]}$ and the restricted valuation $|\ |_x$ on $k[t]$ is the infimum of the translated Gauss valuations $|\ |_{a_i,r_i}$. Set $t'=t\otimes 1-1\otimes t$, then $t'=(t-a_i)\otimes 1-1\otimes(t-a_i)$ and hence $|t'|\le r_i$, yielding that $|t'|\le r$. In fact, it is easy to see that the exact equality holds. For instance, note that the intersection $B=\cap_i B_i$ has a $K$-point, hence the base change $B\wtimes_kK$ is just a $K$-disc of radius $r$, that is, $K\wtimes_kK=K\{t'\}_r$. If $r\notin|k|$ we obtain that $K\otimes_kK$ is not strict, and if $r=1$ we have that $t'$ lies in $(K\otimes_kK)^\circ$, but not in $\Kcirc\otimes_{\kcirc}\Kcirc$.
\end{exam}

\subsection{Descent}
There is a standard trick in Berkovich geometry which reduces many questions to the strict case -- apply a base change with respect to an appropriate extension $K/k$ with a large $|K^\times|$ (typically a Gauss extension) and descend the results. We will only need to use it when $k$ is discretely valued and Lemma~\ref{reeslem}(iii) cannot be applied as it is.

\begin{lem}\label{descent}
Assume that $k$ is real valued and $V$ is a seminormed $k$-vector space. Let $k_r=k(t)$ be the Gauss extension with $|t|=r$. Then $V\into V_r=V\otimes k_r$ is an isometry and vectors $v_1\..v_n$ are $s$-orthogonal over $k$ if and only if their images in $V_r$ are $s$-orthogonal over $k_r$.
\end{lem}
\begin{proof}
One checks by a direct inspection that $V\into V\otimes k[t]$ is an isometry and $v_1\..v_n$ are $s$-orthogonal over $k[t]$. The claim follows easily.
\end{proof}

\begin{rem}
In principle, instead of using the descent trick one could use a more technical approach which treats the discrete and non-discrete cases on an equal footing by considering the balls of all radii at once. The {\em Rees algebra} $k^\circ_\gr=\oplus_{r>0}\kcirc_{\le r}$ is a graded valuation ring of the graded field $\oplus_{r>0} k$ and its graded residue field is $\tilk_\gr$. A similar construction applies to seminormed $k$-vector spaces. If the valuation is not discrete, $V^\circ_\gr$ contains almost the same information as $\Vcirc$, while in the discrete case it is essentially controlled by $\tilV_\gr$. Each $s\le 1$ defines a homogeneous ideal $\cA^\circ_\gr(s)=\oplus_{r>0}\cAcirc_{rs}$, and the {\em $s$-thick graded reduction} $\cA^\circ_\gr/\cA^\circ_\gr(s)=\oplus_{r>0}\cAcirc_r/\cAcirc_{rs}$ of $\cA$ we will use in \S5. When $s=|\pi|\in|k^\times|$ this is just $\cA^\circ_\gr/\pi\cA^\circ_\gr$.
\end{rem}

\subsection{Base change of isometries}
As a toy application let us reprove the following simple well known statement, e.g. see \cite[Lemme~3.1]{Poineau}.

\begin{lem}\label{isomlem}
Let $U,U',V$ be seminormed vector spaces over a real valued field $k$, and let $U\into U'$ be an isometry. Then $U\otimes V\into U'\otimes V$ is an isometry.
\end{lem}
\begin{proof}
Set $W=U\otimes V$ and $W'=U'\otimes V$. If the valuation is not discrete, then $U'^\circ/\Ucirc$ is torsion free by Lemma~\ref{reeslem}(i), and hence $$(U'^\circ\otimes_{\kcirc}\Vcirc)/(\Ucirc\otimes_{\kcirc}\Vcirc)=(U'^\circ/\Ucirc)\otimes_{\kcirc}\Vcirc$$ is torsion free by Lemma~\ref{torlem} below. In view of Lemma~\ref{reeslem}(iii) this implies that also $W'^\circ/\Wcirc$ is almost torsion free and hence $W\into W'$ is an isometry by Lemma~\ref{reeslem}(i).

The case of a discrete valuation can be reduced to the above by tensoring with $k_r$, where $r\notin|k^\times|^\QQ$, and using Lemma~\ref{descent}. The same argument reduces the trivially valued case to the discretely valued one.
\end{proof}

\begin{lem}\label{torlem}
Assume that $M,N$ are $\kcirc$-modules and $N$ is torsion free. If the torsion of $M$ is killed by $\pi$ (resp. almost vanishes, resp. vanishes), then the same is true for $M\otimes_{\kcirc}N$.
\end{lem}
\begin{proof}
This easily follows from the standard fact that a $\kcirc$-module is flat if and only if it is torsion free.
\end{proof}

\subsection{Universality of $r$-orthogonality}
As another application let us prove the following simple result, which of course admits more straightforward proofs too.

\begin{lem} \label{rlem}
Assume that $l/k$ is an extension of real valued fields and $U$ is a seminormed $k$-vector space. Then elements $u_1\..u_n\in U$ are $r$-orthogonal for $r\in(0,1]$ if and only if their images $u_1\..u_n\in U_l$ are $r$-orthogonal over $l$.
\end{lem}
\begin{proof}
We can assume that the valuation is not discrete since the other cases reduce to this by tensoring with some $k_s$ once or twice. Then by Lemma~\ref{reeslem}(ii) the map $\phi\:W=\oplus_{i=1}^n ke_i\to U$ taking $e_i$ to $u_i$ is injective and the torsion of the cokernel $C$ of $\phicirc\:\Wcirc\to\Ucirc$ is killed by any $\pi$ with $|\pi|\le r$. Tensoring with $l$ yields a morphism $\phi_l\:W_l=\oplus_{i=1}^n le_i\to U_l$, and we set $C_l=\Coker(\phi_l^\circ)$

$$\xymatrix{
0\ar[r]& \Wcirc_l\ar[r]^{\phicirc_l}& \Ucirc_l \ar[r]& C_l\ar[r]& 0\\
0\ar[r]&  \Wcirc\otimes_{\kcirc}\lcirc\ar[r]^{\phicirc\otimes_{\kcirc}\lcirc}\ar@{^{(}->}[u]& \Ucirc\otimes_{\kcirc}\lcirc\ar[r]\ar@{^{(}->}[u]& C\otimes_{\kcirc}\lcirc\ar[r]\ar[u]^\psi& 0.
}$$
Since $\Wcirc_l/(\Wcirc\otimes_{\kcirc}\lcirc)$ and $\Ucirc_l/(\Ucirc\otimes_{\kcirc}\lcirc)$ almost vanishes by Lemma~\ref{reeslem}(iii), the same is true for $\Ker(\psi)$ and $\Coker(\psi)$. By Lemma~\ref{torlem} $(C\otimes_{\kcirc}\lcirc)_\tor$ is killed by any $\pi$ with $|\pi|<r$, hence the same is true for the torsion of $C_l$, and $u_1\..u_n$ are $r$-orthogonal over $l$ by Lemma~\ref{reeslem}(ii).
\end{proof}

As a corollary we obtain the following criterion for failure of geometric multiplicativity. Consider a henselian real valued field $k$ (so, the valuation extends to $k^a$ uniquely). For a real valued extension $l/k$ and an element $t\in l$ by $d_{t/k}=\inf_{c\in k}|t-c|$ we denote the distance between $t$ and $k$. In particular, if $k$ is henselian and $\alp\in k^a$ is algebraic, then the valuation extends to $k^a$ uniquely and $d_{\alp/k}$ is well defined.

\begin{cor}\label{rcor}
Assume that $l/k$ and $K/k$ are extensions of real valued henselian fields such that $L=K\otimes l$ is a field and there exists $\alp\in l$ which is algebraic over $k$ and satisfies the inequality $d_{\alp/K}<d_{\alp/k}$. Then the product norm on $L$ is not multiplicative.
\end{cor}
\begin{proof}
We claim that the pair of elements $1,\alp\in l$ is $r$-orthogonal over $k$ if and only if $r\le r_0=d_{\alp/k}/|\alp|$. Indeed, any linear combination $c\alp+a$ with $c\neq 0$ can be rescaled by $c^{-1}$, hence checking that $1,\alp$ are $r$-orthogonal reduces to testing only linear combinations of the form $\alp-a$, which makes our claim obvious.

Now let us prove the claim. By Lemma~\ref{isomlem} it suffices to prove that $K\otimes k(\alp)$ is not multiplicative, hence we can replace $l$ by $k(\alp)$. Choose any $r$ with $d_{\alp/K}<r|\alp|<d_{\alp/k}$. Then $1,\alp$ are $r$-orthogonal over $k$, but not over $K$ with respect to the valuation of $L$. Therefore Lemma~\ref{rlem} implies that the tensor norm on $L$ is not a valuation.
\end{proof}

For a future use we also record the following simple result.

\begin{cor}\label{dcor}
Let $l/k$ and $K/k$ be two real valued extensions such that the tensor norm on $l\otimes_kK$ is multiplicative and hence extends to a valuation on the field of fractions. Then for any $t\in l$ the distance $d_{t/k}$ in $l$ equals the distance $d_{t/K}$ in $L$.
\end{cor}
\begin{proof}
We can assume that $t\neq 0$. Note that $1,t$ are $r$-orthogonal in $l$ for $r<d_{t/k}/|t|$ and are not $r$-orthogonal for $r>d_{t/k}/|t|$. The same characterization holds for $d_{t/K}$ and it remains to use that the $r$-orthogonality is preserved by base field extensions by Lemma~\ref{rlem}.
\end{proof}

\subsection{Criteria of geometric multiplicativity}
We can also use $r$-orthogonality to provide a criterion when the multiplicativity is preserved by a base change.

\begin{lem}\label{mullem}
Let $l/k$ and $K/k$ be extensions of real valued fields. Assume that $L=K\otimes l$ is a field and $[l:k]=n<\infty$. Provide $L$ with an extension $|\ |$ of the valuation of $K$ and assume that for any $r\in(0,1)$ there exists a basis $\alp_1\..\alp_n\in l$ which is an $r$-orthogonal basis of $L$ over $K$ with respect to the valuation of $L$. Then $L/K$ is unibranched and the tensor norm $\|\ \|$ of $K\otimes l$ coincides with $|\ |$. In particular, $\|\ \|$ is multiplicative.
\end{lem}
\begin{proof}
Since $L$ is weakly cartesian, $[\hatL:\hatK]=[L:K]$ and hence $L/K$ is unibranched. This already shows that $K\otimes l$ is spectrally multiplicative and $\|\ \|_\sp=|\ |$. Assume to the contrary that the tensor norm $\|\ \|$ of $L$ is not multiplicative. Take any $x\in L$ with $\|x\|>|x|$ and choose $r<1$ so that $r\|x\|>|x|$. Let $\alp_1\..\alp_n\in l$ be an $r$-orthogonal basis of $L/K$ and present $x$ as $x=\sum c_i\alp_i$, where $c_i\in K$. Then $|x|\ge r\max(|c_i|\cdot|\alp_i|)\ge r\|x\|$, a contradiction.
\end{proof}


This general criterion can be made quite explicit in many practical situations. One such case was already established in Lemma~\ref{grlem}. Here is a subtler case when defect can happen. In fact, the claim below is not covered by Lemma~\ref{grlem} precisely when $l/k$ has defect.

\begin{cor}\label{mulcor}
Let $l/k$ and $K/k$ be extensions of real valued fields. Assume that $k$ and $K$ are henselian, $L=K\otimes l$ is a field, $[l:k]=p=\cha(\tilk)$, and $\alp\in l$ is such that $d_{\alp/k}>0$. Then the tensor norm on $L$ is multiplicative if and only if $d_{\alp/k}=d_{\alp/K}$.
\end{cor}
\begin{proof}
The opposite implication is covered by Corollary~\ref{rcor}, so assume that $d_{\alp/k}=d_{\alp/K}$. If the valuation of $k$ is discrete, then replacing $\alp$ by $\alp-c$ with $c\in k$ we can assume that $|\alp|=d_{\alp/K}$. In this case, $1,\alp\..\alp^{p-1}$ is an orthogonal basis of $K(\alp)/K$, hence $L=K(\alp)$ and then the tensor norm is the valuation by Lemma~\ref{mullem}.

Assume now that the valuation is not discrete. Then there exists a sequence $\alp_i=(\alp-c_i)/\pi_i$ with $c_i,\pi_i\in k$, such that $|\alp-c_i|$ tends to $d_{\alp/K}$ and $|\alp_i|$ monotonically increases and tends to 1. By \cite[Theorem~4.3.4(i)]{KT}, for any $r<1$ there exists $i$ such that $1,\alp_i\..\alp_i^{p-1}$ is $r$-orthogonal (in the tamely closed case this is also covered by \cite[Proposition~6.3.13]{GRbook} and, in fact, the assumption that $K^t=K$ can be easily eliminated there too). Now, the claim follows from Lemma~\ref{mullem}.
\end{proof}

We can summarize this in the following example.

\begin{exam}\label{mulexam}
(i) As in Corollary~\ref{mulcor} assume that $k$ and $K$ are henselian and $l=k(\alp)$ is wild, separable of degree $p$ over $k$. Let $r_{\alp/k}=\min_i|\alp-\alp_i|$ denote the minimal distance between $\alp$ and other roots of its minimal polynomial (in fact, they all are equidistant). Note that $r_{\alp/k}\le d_{\alp/k}$ by Krasner's lemma and $l/k$ is almost unramified if and only if $r_{\alp/k}=d_{\alp/k}$.

(a) The tensor norm $\|\ \|$ is spectrally multiplicative if and only if $L=K\otimes l$ is a field. This is the case when $r_{\alp/k}<d_{\alp/K}$ and this fails when $d_{\alp/K}<r_{\alp/k}$.

(b) By Corollary \ref{mulcor} the tensor norm is multiplicative if and only if $L$ is a field and $d_{\alp/k}=d_{\alp/K}$. In particular, it is not spectral whenever $r_{\alp/k}<d_{\alp/K}<d_{\alp/k}$.

(c) More generally it is easy to see that the tensor norm is spectral if and only if either $r_{\alp/k}=d_{\alp/k}$ or $d_{\alp/K}=d_{\alp/k}$.

(ii) Let us explain that the mechanism for the failure of geometric spectrality is the same as with geometric non-reducedness in commutative algebra. For simplicity we assume that $p=\cha(k)$, as the mixed characteristic case requires some estimates with $|p|$. A separable extension $l/k$ of degree $p$ is not almost unramified if and only if there exists $\alp\in l$ and $a\in k$ such that $|\alp^p-a|<r=\inf_{c\in k}|a-c^p|$, e.g. see \cite[Proposition~6.1.4]{temst}. If $r_{\alp/k}<d_{\alp/K}<d_{\alp/k}$, then there exists $\alp'\in K$ such that $|\alp'-\alp|<d_{\alp/k}$. Then by Lemma~\ref{rlem} $\|\alp-\alp'\|=d_{\alp/k}$, but $\|(\alp-\alp')^p\|=\|\alp^p-\alp'^p\|<(d_{\alp/k})^p$.

(iii) Assume that $k$ is complete. Since up to tame extensions any wild extension can be split into composition of extensions of degree $p$, the above implies that a finite extension $l/k$ with a henselian $k$ is geometrically spectral if and only if it is almost tame (see also \cite[\S6.3.6]{topforms}).
\end{exam}

\section{Geometric spectral multiplicativity}
In this section, we conclude our study of geometric spectral multiplicativity.

\subsection{Transcendental extensions}
Consider the following valuative analogue of purely transcendental extensions in commutative algebra. A $k$-valuation $|\ |_l$ on $l=k(t)$ is called {\em almost $k$-split} if its restriction onto $k[t]$ is an infimum of translated Gauss valuations $|\ |_{a_i,r_i}$, where $|\ |_{a,r}$ is defined by $|\sum_jc_j(t-a)^j|_{a,r}=\max_j r^j|c_j|$. If $k$ is algebraically closed, then $k(t)$ is automatically almost split. For example, if $d_{t/k}>0$, then this follows from the classification of points on Berkovich affine line over $\hatk$ in \cite[\S1.4]{berbook}. Otherwise $l\subset\hatk$, and the image $a\in\hatk$ of $t$ defines a classical point $|\ |_{a,0}$ of $\bbA^1_\hatk$, which is the semivaluation on $\hatk[t]$ with kernel $(t-a)$, and of course $|\ |_{a,0}$ is also an infimum of Gauss valuations $|\ |_{a,r}$ dominating it. The restriction of $|\ |_{a,0}$ onto $k[t]$ is the valuation induced from $l$.

\begin{lem}\label{translem}
Assume that $l=k(t)$ is a real valued field whose valuation is almost $k$-split. Then a seminormed $k$-algebra $\cA$ is spectral, spectrally multiplicative or multiplicative if and only if the seminormed $\cA_l$ is so.
\end{lem}
\begin{proof}
The claim for spectral multiplicativity follows from the claim for multiplicativity, because we can replace $\cA$ with its spectralization by Lemma~\ref{prelem}(ii). Furthermore, the tensor seminorm on $\cA_l$ and its associated spectral seminorm are determined by their restrictions onto $\cA[t]=\cA\otimes_kk[t]$ via the rule $\|a/x\|=\|a\|/|x|$, where $a\in\cA[t]$ and $0\neq x\in k[t]$. Once we work with $k[t]$, it suffices to deal with the case when $l$ is provided with a Gauss valuation because multiplicativity, power-multiplicativity and tensoring with a seminorm are preserved under limits of valuations. In this case the multiplicativity is just the Gauss lemma, and power-multiplicativity is proved similarly but let us sketch the argument for completeness. Translate the Gauss valuation so that 0 is a center. Assume that $\cA$ is power-multiplicative, and let us show that $|x^n|=|x|^n$ for $x=\sum_ia_it^i\in\cA[t]$. Choose the minimal $j\in\NN$ such that $|a_jt^j|=\max_i|a_it^i|=|x|$. Then $x^n=\sum_ib_it^i$ with $|b_{nj}|=|a_j|^n$, which implies the claim.
\end{proof}

\begin{rem}
Assume that $\cA$ as above is normed. It is easy to see that if $d_{t/k}>0$, then $\cA_l$ is normed as well, but for $l=k(t)\subset\hatk$ the seminorm on $l\otimes_kl$ is not a norm because $\|t\otimes 1-1\otimes t\|=0$.
\end{rem}

\subsection{Spectral multiplicativity}
Next we study how spectral multiplicativity behaves in the case of algebraic ground field extensions.

\begin{lem}\label{premullem}
Let $l/k$ be a finite extension of real valued fields, and let $\cA$ be a spectrally multiplicative seminormed $k$-algebra with $K=\Frac(\cA^\sp)$. Then the following conditions are equivalent:
\begin{itemize}
\item[(i)] The base change $\cA_l$ is spectrally multiplicative.
\item[(ii)] The reduction of $K^h\otimes_{k^h} l^h$ is a field.
\item[(iii)] The reduction of $\hatK\otimes_\hatk\hatl$ is a field.
\end{itemize}
\end{lem}
\begin{proof}
By Corollary \ref{precor} $\cA_l$ is spectrally multiplicative if and only if $K\otimes l$ is. The completion preserves the analytic spectrum, hence the latter happens if and only if $\wh{K\otimes l}=\hatK\otimes_{\hatk}\hatl$ is spectrally multiplicative (there is no need to complete the tensor product because $\hatl/\hatk$ is finite). It remains to note that $\cM(\hatK\otimes_{\hatk}\hatl)=\coprod_i\cM(\hatK_i)$, where $(\hatK\otimes_\hatk\hatl)^\red=\prod_i\hatK_i$. The claim about henselizations follows because they are separably closed in the completions.
\end{proof}

A couple of words about the reasons which forced us to choose the oddly looking formulation of the lemma.

\begin{rem}
(i) The lemma asserts that $\cA$ is not universally spectrally multiplicative if and only if $\Spec(K^h)$ is not geometrically irreducible over $k^h$. One cannot formulate a more global criterion, as $\Spec(\cA)$ itself can be geometrically irreducible (even when $k$ is complete and $\cA$ is affinoid).

(ii) We formulated (ii) and (iii) using henselizations and completions because a similar claim fails for $K\otimes l$ itself. For example, if $K=l$ is a quadratic subextension of $k^h/k$, then $l\otimes l=l\times l$, but this ring is spectrally multiplicative because the tensor norm has a kernel and $(l\otimes l)^\sp=l$ (e.g. use that the completion is $\hatk=\hatl$).
\end{rem}

\subsection{Main theorem}
At this stage we can already prove the main result about geometric spectral multiplicativity. This is very similar to the usual theory of geometric irreducibility, up to the nuance with the completion or henselization.

\begin{theor}\label{multiplicativeth}
Let $k$ be a real valued field and let $\cA$ be a spectrally multiplicative seminormed $k$-algebra with $K=\Frac(\cA^\sp)$. Then $\cA$ is geometrically spectrally multiplicative if and only if $\hatk$ is separably closed in $\hatK$ (or $k^h$ is separably closed in $K^h$). In particular, $k^h$ is separably closed if and only if any spectrally multiplicative $k$-algebra is geometrically spectrally multiplicative.
\end{theor}
\begin{proof}
By Corollary \ref{precor} it suffices to consider the case when $\cA=K$. Lemma~\ref{premullem} easily implies that if $k^h$ is not separably closed in $K^h$, then $K$ is not geometrically spectrally multiplicative over $k$. Conversely, assume that $k^h$ is separably closed in $K^h$. We should prove that $K\otimes l$ is spectrally multiplicative for a given real valued extension $l/k$. It suffices to prove this for a real valued extension $l'$ of $l$ because $K\otimes l$ embeds isometrically into $K\otimes l'$. Thus, replacing $l$ with $l^a$ we can assume that $l$ is algebraically closed.

The following fact follows from compatibility of henselization with filtered colimits and will be tacitly used in the sequel: if $l$ is the filtered colimit of its subfields $l_i$ such that each $K\otimes l_i$ is spectrally multiplicative, then also $K\otimes l$ is spectrally multiplicative. As a first application, combining it with Lemma~\ref{premullem} we obtain that $K\otimes k^a$ is spectrally multiplicative. Therefore we can replace $k$ and $K$ by $k^a$ and $(K\otimes k^a)^\sp$, achieving that $k$ is algebraically closed.

Next, we claim that for any $t\in l$ the seminormed ring $K\otimes k(t)^a$ is spectrally multiplicative. Recall that $K\otimes k(t)$ is spectrally multiplicative by Lemma~\ref{translem}, so $(K\otimes k(t))^\sp$ is a domain and we denote the fraction field by $L$. We can assume that $t\notin\hatk$ as otherwise the completion is just $\hatK$ and there is nothing to prove. Otherwise, $L=K(t)$ and the argument from the previous paragraph would allow us to conclude that also $K\otimes k(t)^a$ is spectrally multiplicative once we show that $\wh{k(t)}$ is separably closed in the completion $\wh{K(t)}$. We postpone this to Lemma~\ref{aclem} below.

Finally, each $K\otimes k(t_1\..t_n)^a$ with $t_1\..t_n\in l$ is spectrally multiplicative by the above claim and the induction on $n$, and hence $K\otimes l$ is also spectrally multiplicative by the colimit argument.
\end{proof}

It remains to establish the result we have used above. Despite a na\"ively looking formulation we do not see a short elementary proof and prefer to use quite deep results about one-dimensional extensions of real valued fields.

\begin{lem}\label{aclem}
Let $K/k$ be an extension of real valued fields and assume that $k$ is algebraically closed. Let $k(t)$ be provided with a $k$-valuation such that $d_{t/k}>0$ and let $K(t)=\Frac(K\otimes k(t))$ be provided with the tensor valuation. Then $l=\wh{k(t)}$ is algebraically closed in $L=\wh{K(t)}$.
\end{lem}
\begin{proof}
We can replace $k$ and $K$ by their completions since this does not change $l$ and $L$. We start with the Abhyankar case. Assume that $r=\inf_{c\in k}|t-c|$ is attained. Then translating $t$ we can assume that actually $|t|=r$, and hence $l=\wh{k_r}$ is the completed Gauss extension of $k$ and then also $L=\wh{K_r}$ is the completed Gauss extension of $K$. Therefore $\till_\gr=\tilk_\gr(\tilt)_r$ (the graded fraction field of $\tilk_\gr[\tilt]_r$ obtained by inverting all non-zero homogeneous elements) is algebraically closed in $\tilL_\gr=\tilK_\gr(\tilt)_r$. On the other hand, $l$ is defectless by the stability theorem, e.g. see \cite[Corollary~6.3.6]{temst}. So, any non-trivial finite extension $l'/l$ induces a non-trivial finite extension of the graded residue field $\till_\gr$ and hence $l'$ cannot be contained in $L$.

Assume now that the infimum is not attained and hence $l/k$ is of type 4. We should obtain a contradiction to the following assumption: $l'/l$ is finite with $d=[l':l]>1$, but $l'\subset L=\wh{Kl}$. In particular, we can enlarge $K$ and assume that $l'$ admits a $k$-embedding into it. By \cite[Corollary~2.1.10]{temst} $l'=\wh{k(x)}$ is also topologically generated by a single element, hence it is an extension of $k$ of type 4 and $L=\wh{Kl'}=\wh{K(x)}$. Furthermore, if $t'\in l'$ is such that $|t-t'|<{d}_{t/K}\le {d}_{t/k}$ (in fact, the second inequality is an equality), then one can replace $t$ and $l$ by $t'$ and $\wh{k(t')}$ because $[l':\wh{k(t')}]=[l':l]$ and $L=\wh{K(t')}$ by \cite[Lemmas~6.3.3 and 6.3.2]{temst}.

We will now use some basic results about extensions of type 4 recorded in \cite[Lemma~3.2.4]{small}. Since $k[x]$ is dense in $l'$ by claim (i) of that lemma, we can replace $t$ by an appropriate approximation, which is a polynomial in $x$. So, we assume in the sequel that $t\in k[x]$. Choose $a_i\in k$ such that the sequence $r_i=|x_i|$ with $x_i=x-a_i$ decreases and tends to $r=d_{x/k}$. By claim (ii) of the cited lemma the absolute values of the coefficients $a_{i,j}$ in the presentation $t=\sum_j a_{i,j}x_i^j$ stabilize for $i\gg 0$, so we set $s_j=|a_{i,j}|$ for $i\gg 0$. Finally, by claim (iv) of the cited lemma $d=[\wh{k(x)}:\wh{k(t)}]$ is determined by the following conditions: $s_dr^d=\max_{j>0}s_jr^j$ and $s_dr^d>s_jr^j$ for $j>d$.

By our assumption, $K$ contains a copy of $l'$, hence contains an element $a$ such that $|a-a_i|=r_i$ for any $i$. Set $x_\infty=x-a$, then $|x_\infty|=r=d_{x/K}=d_{x_\infty/K}$, where the second equality holds by Corollary~\ref{dcor}. It follows that $l'\wtimes_kK=K\{x_\infty\}_r$, and hence $L=\wh{K(x_\infty)_r}$ is the completed Gauss extension of $K$ with coordinate $x_\infty$. In particular, $\tilL_\gr=\tilK_\gr(\tilx_\infty)_r$. Now, let us do the same computation for the presentation $L=\wh{K(t)}$. The presentation $t=\sum b_jx_\infty^j$ also satisfies $|b_j|=s_j$ and hence $b_dx_\infty^d$ is the dominant term of $t_b=t-b_0$ of maximal degree. In particular, $t_b$ is orthogonal to $K$ and $|t_b|=s_dr^d$, which yields the interpretation $L=\wh{K(t_b)_{s_dr^d}}$ as another completed Gauss extension. In particular, we obtain that $\tilK_\gr(\tilx_\infty)_r=\tilL_\gr=\tilK_\gr(\tilt_b)_{s_dr^d}$, which cannot be true because $\tilt_b$ is a polynomial in $\tilK_\gr[\tilx_\infty]_r$ of degree $d>1$.
\end{proof}

\section{Geometric spectrality}
It remains to provide criteria of geometric spectrality. The main case to deal with is when $p=\cha(\tilk)>1$. We will see that otherwise any spectral algebra is geometrically spectral.

\subsection{$p$-versal extensions}
Let $k$ be a real valued field of residual exponential characteristic $p$. The reader can assume that $p>1$ as otherwise the discussion becomes vacuous. We say that $a\in k$ is {\em $p$-regular} if $|pa|<\inf_{c\in k}|a-c^p|$. We set $|a|_{k/p}=\inf_{c\in k}|a-c^p|$ if $a$ is $p$-regular, and $|a|_{k/p}=0$ otherwise. The latter correction is only needed to make things work smoother in the mixed characteristic, while in the equal characteristic case $|\ |_{k/p}$ is the residue seminorm on the group $k/k^p$, and $a$ is $p$-regular if and only if $a\notin\hatk^p$.

An extension of real valued fields $l/k$ will be called {\em $p$-versal} if there exists $0<r<1$ such that $|a|_{l/p}\le r|a|_{k/p}$ for any $p$-regular element $a\in k$. For any such $r$ we say that $l/k$ is $p$-versal of {\em thickness} $r$.

\begin{rem}
\label{rem:p-root}
(o) If $p=1$ (i.e. the residual characteristic is zero), there are no $p$-regular elements and any extension is $p$-versal.

(i) If $\cha(k)=p$, then $k^{1/p}/k$ is $p$-versal of unbounded thickness (i.e. any positive $r$ works). An $r$-thick extension should be considered as an extension which contains a deformation of $k^{1/p}$, though this is an analogy only, and it can freely happen, for example, that $l/k$ is purely transcendental.

(ii) Assume that $|p|<r<1$. Then $l/k$ is $p$-versal of thickness $r$ if and only if for any $s\in(r,1)$ the $s$-thick graded reduction $\kcirc_\gr/\kcirc_\gr(s)$ lies in the image of the Frobenius on $\lcirc_\gr/\lcirc_\gr(s)$. If $l$ is discretely valued, this is also true for $s=r$ and it suffices to test $s=r$ for $p$-versality.

(iii) Let us say that $l/k$ is {\em weakly $p$-versal} if $|a|_{l/p}<|a|_{k/p}$ for any $p$-regular $a\in k$. If $l/k$ is weakly $p$-versal, then the graded reduction $\tilk_\gr$ lies in the image of the Frobenius on $\till_\gr$. The converse holds when $k$ is defectless.
\end{rem}

The notion of weak $p$-versality is not really useful and does not catch properties one expects from a deformation of $k^{1/p}/k$ when $\cha(k)=p$. It can even happen that $k$ is complete and admits a weakly $p$-versal extension of degree $p$, but $[k:k^p]=p^2$. Here is an example, whose justification is left to the interested reader.

\begin{exam}
Let $F$ be complete non-trivially valued algebraically closed of characteristic $p$, provide $E=\wh{F(t)}$ with the Gauss valuation and choose a sequence $c_i\in F$ with $r_i=|c_i|$ strictly decreasing and tending to $r>0$. Set $f_n=\sum_{i=0}^nc_it^i$ and let $k=\wh{E(y)}$ be the almost split extension of type 4 given by $|y-f_n|=r_{n+1}$. Then $[k:k^p]=p^2$ and $l=k(t^{1/p})/k$ is weakly $p$-versal. An equivalent observation is that this example reveals the following pathological behaviour: $k$ is not stable, e.g. $k^{1/p}/k$ has a nontrivial defect, but $k$ has no non-trivial purely inseparable immediate extensions (and even no non-trivial algebraic immediate extensions when $r\notin|F|$), and any wild extension of degree $p$ extends the residue field.
\end{exam}

\subsection{The key lemma}
Our main result concerning the geometric spectrality is that it can be tested on a single $p$-versal extension $l/k$. We will use the $p$-versality assumption through the following key lemma.

\begin{lem} \label{keylem}
Let $k$ be a real valued field of residual exponential characteristic $p$, let $\cA$ be a normed $k$-algebra and let $e_1,\ldots,e_n\in\cA$ be $s$-orthogonal elements, where $|p|<s^p<1$. Assume that there exists a $p$-versal extension $l/k$ of thickness $r<s^p$ such that the base change $\cA_l$ with the tensor norm $\|\ \|$ is spectral. Then the elements $e_1^p,\ldots,e_n^p$ are $s^p$-orthogonal.
\end{lem}
\begin{proof}
We need to prove that for any $k_1,\ldots,k_n\in k$ the element $x=\sum_{i=1}^nk_ie_i^p$ satisfies the inequality $\|x\|\ge s^p\rho$, where $\rho=\max\limits_{1\le i\le n}\|k_ie_i^p\|$.

Let us first illustrate the idea in the model case, when $\cha(k)=p$ and $l=k^{1/p}$. Here we can simply use that $e_1\..e_n$ are $s$-orthogonal in $\cA_l$ by Lemma~\ref{rlem}, and hence $x^{1/p}=\sum_{i=1}^nk_i^{1/p}e_i$ satisfies the inequality $\|x^{1/p}\|\ge s\max\|k_i^{1/p}e_i\|=s\rho^{1/p}$. The claim follows since $\|x\|=\|x^{1/p}\|^p$ by the spectrality of $\cA_l$.

The same method works in general, once we use the $p$-versality to choose approximate roots $l_i\in l$ and check that the $p$-th power is additive up to error terms which are below our thresholds. Indeed, choose $l_i$ such that $|l^p_i-k_i|<s^p|k_i|$ if $k_i\neq 0$ and $l_i=0$ otherwise, and set $y=\sum_{i=1}^nl_ie_i$. In particular, $\|y\|\ge s\rho^{1/p}$ and then $\|y^p\|\ge s^p\rho$ by the spectrality. On the other hand, $\|x-\sum_il_i^pe_i^p\|<s^p\rho$ and
$$\left\|y^p-\sum_il_i^pe_i^p\right\|\le\max_i\|pl_i^pe_i^p\|=\rho|p|<s^p\rho$$
hence also $\|x\|=\|y^p\|\ge s^p\rho$, as claimed.
\end{proof}

\subsection{Main theorem}
Now we can prove our main result concerning geometric spectrality.

\begin{theor}\label{spectralth}
Let $k$ be a real valued field of residual exponential characteristic $p$ and let $\cA$ be a weakly cartesian normed $k$-algebra. Then $\cA$ is geometrically spectral if and only if there exists a $p$-versal extension $l/k$ such that $\cA_l$ is spectral.
\end{theor}
\begin{proof}
The direct implication is obvious. Conversely, assume that $l/k$ is $p$-versal of thickness $r$ and $\cA_l$ is spectral. Let $K/k$ be any extension of real valued fields. We need to prove that $\cA_K$ is also spectral. Assume first that $p=1$. As in the proof of Theorem~\ref{multiplicativeth}, we can replace $K$ by a larger field and it suffices to prove the claim for a cofinal family of subextensions of $K$. Therefore the claim reduces to the two cases: (a) $K/k$ is finite, (b) $K=k(t)$ and $k$ is algebraically closed. The second case is covered by Lemma~\ref{translem}. In the first case, note that it suffices to prove spectrality of the completion $\wh{\cA_K}=\hatcalA\otimes_\hatk\hatK$. This reduces the claim further to the case of a complete ground field $k$, and then it follows from Lemma~\ref{grlem}(iii) because $K/k$ is defectless and $\tilk_\gr$ is a graded field of characteristic zero and hence reduced graded $\tilk_\gr$-algebras are geometrically reduced.

Assume now that $p>1$. In this case it is enough to prove that any $x\in\cA_K$ satisfies $\|x^p\|=\|x\|^p$, because then also $\|x^{p^n}\|=\|x\|^{p^n}$ and hence the norm is power-multiplicative. Assume by contradiction that $\|x^p\|=\gamma\|x\|^p$, where $\gamma<1$. Fix an orthogonality threshold $s<1$ such that $\max\{\gamma^{1/p}, r^{1/p},|p|^{1/p}\}<s$. Choose a finite-dimensional subspace $V\subseteq\cA$ such that $x\in V_K$. By our assumption and Lemma~\ref{cartlem}, $V$ possesses an $s$-orthogonal basis $(e_1,\ldots,e_n)$.

Consider the presentation $x=\sum_{i=1}^n a_ie_i$ with $a_i\in K$ and let $\rho=\max\|a_ie_i\|$. Since $e_1\..e_n$ are $s$-orthogonal over $K$ by Lemma~\ref{rlem}, $\|x\|\ge s\rho$. Next, consider the element $y=\sum_{i=1}^na_i^pe_i^p$ and note that $\|x^p-y\|\le\rho^p|p|$. On the other hand, the elements $e_1^p\..e_n^p\in\cA$ are $s^p$-orthogonal over $k$ by Lemma~\ref{keylem} and using Lemma~\ref{rlem} once again we obtain that $\|y\|\ge s^p\rho^p>\rho^p|p|$. This yields a contradiction to the choice of $\gamma$ since $$\|x^p\|=\|y\|\ge s^p\rho^p\ge s^p\|x\|^p>\gamma\|x\|^p=\|x^p\|.$$
\end{proof}

Recall that $k$ is perfectoid if it is non-discretely valued, complete, of positive residual characteristic $p$ and the Frobenius is surjective on $\kcirc/\pi\kcirc$ for some (and then any) $\pi\in\kcirccirc$ with $|p|\le|\pi|$. Thus, a non-trivially valued $k$ is $p$-versal over itself if and only if either $p=1$ or $\hatk$ is perfectoid, and we obtain the following corollary.

\begin{cor}
Let $\cA$ be a spectral normed algebra over a real valued field $k$ of residual exponential characteristic $p$. Assume that either $p=1$ or $\hatk$ is perfectoid. Then $\cA$ is geometrically spectral.
\end{cor}
\begin{proof}
We should prove that $\cA\otimes_kK$ is spectral for a real valued extension $K/k$. It suffices to prove that the completion $\cA\wtimes_k K$ is spectral. But this is so because $\hatcalA$ is spectral and $\hatk/\hatk$ is $p$-versal by our assumptions, so $\hatcalA\otimes_\hatk\hatK$ is spectral.
\end{proof}


\begin{rem}
The following notion was introduced by Ducros in \cite[Definition~1.6]{Ducrosexcel} for complete fields and it plays the role of separable extensions in analytic geometry. An extension $l/k$ of complete real valued fields is called {\em analytically separable} if either ${\rm char}(k)=0$ or the tensor norm of $l\wtimes_kk^{1/p}$ is equivalent to the seminorm induced from $l^{1/p}$. For example, if $k$ is not perfect and $t\in\widehat{k^{1/p^\infty}}$ is transcendental over $k$, then $\wh{k(t)}/k$ is not analytically separable, $(k(t){\widehat{\otimes}}_kk(t))^\sp=\widehat{k(t)}$ and the spectral seminorm vanishes on $1\otimes t-t\otimes 1$.

For comparison: $l/k$ is geometrically spectral when the tensor norm and the spectral norm on $l\otimes_k l'$ coincide for some (and then any) $p$-versal extension $l'/k$, and this property can fail only when the residual characteristic is positive. While analytic separability only asks for the seminorms to be equivalent, i.e. $r\|\ \|_\otimes\le\|\ \|_\sp$ for some $r>0$, it has to be tested on $l'=k^{1/p}$ (rather than for an $r$-deformation), and it can fail only when $k$ itself is of positive characteristic.
\end{rem}

\subsection{Geometric multiplicativity}
Combining Theorems \ref{multiplicativeth} and \ref{spectralth} we obtain the following result.

\begin{theor}\label{multth}
Let $k$ be a real valued field of residual exponential characteristic $p$ and let $\cA$ be a weakly cartesian multiplicative normed $k$-algebra. Then $\cA$ is geometrically multiplicative if and only if there exists a $p$-versal extension $l/k$ such that $\cA_l$ is spectral, and $k^h$ is separably closed in $K^h$, where $K=\Frac(\cA)$. In particular, if $k$ is algebraically closed, then $\cA$ is geometrically multiplicative.
\end{theor}

\bibliographystyle{amsalpha}
\bibliography{universal_norms}

\end{document}